\documentclass{amsart}
\usepackage{amsmath}
\usepackage{amssymb}
\usepackage{amsthm}
\usepackage{graphicx}
\usepackage{color}
\usepackage[all]{xypic}

\def\g{\mathfrak g}      
\def\h{\mathfrak h}      
\def\u{\mathfrak u}      
\def\t{\mathfrak t}      
\def\k{\mathfrak k}      
\def\n{\mathfrak n}      
\def\p{\mathfrak p}      
\def\a{\mathfrak a}

\def\m{\mathfrak m}

\def\w{\mathbf w}        

\def\R{\mathbb R}        
\def\C{\mathbb C}        

\def\SU{\mathrm{SU}}     
\def\U{\mathrm{U}}       
\def\SL{\mathrm{SL}}     
\renewcommand{\sl}{\mathfrak{sl}}   

\newtheorem{theorem}{Theorem}[section]
\newtheorem{lemma}{Lemma}[section]
\newtheorem{corollary}{Corollary}[section]
\newtheorem{proposition}{Proposition}[section]

\theoremstyle{definition}

\newtheorem{example}{Example}[section]
\newtheorem*{remarks}{Remarks}

\theoremstyle{remark}
\newtheorem{remark}{Remark}[section]

\numberwithin{equation}{section}

\begin{document}
\title{Noncompact Groups of Inner Type and Factorization}
\author{Arlo Caine}
\email{jacaine@cpp.edu}

\author{Doug Pickrell}
\email{pickrell@math.arizona.edu}

\begin{abstract}
We investigate Birkhoff (or triangular) factorization and (what we propose to call) root subgroup factorization for elements of a noncompact simple Lie group $G_0$ of inner type.  For compact groups root subgroup factorization is related to Bott-Samelson
desingularization, and many striking applications have been discovered by Lu (\cite{Lu}).  In this paper, in the inner noncompact case, we obtain parallel characterizations of the Birkhoff components of $G_0$ and an analogous construction of root subgroup coordinates for the Birkhoff components.  As in the compact case, we show that the restriction of Haar measure to the top Birkhoff component is a product measure in root subgroup coordinates.
\end{abstract}

\maketitle
\{2000 Mathematics Subject Classifications:  22E67\}

\setcounter{section}{-1}


\section{Introduction}

Given a semisimple Lie algebra $\m$ over $\R$, write $\g$ for the complexification $\m^\C$ of $\m$ and let $G$ denote the simply connected complex Lie group with Lie algebra $\g$.  Let $M$ denote the connected subgroup of $G$ with Lie algebra $\m$.  This paper concerns the decomposition of $M$ induced from a Birkhoff decomposition of $G$. More specifically, fix a triangular decomposition $\g=\n^-+\h+\n^+$ (which is compatible with $\m$ in a sense described below).  With $H=\exp(\h)$, $N^\pm=\exp(\n^\pm)$, and $W=N_G(H)/H$, we obtain a Birkhoff decomposition of $G$
\[
G=\bigsqcup_{w\in W}\Sigma_w^G \text{ where }\Sigma_w^G=N^-wHN^+
\]
and we are interested in the induced decomposition of $M$
\[
M=\bigsqcup_{w\in W}\Sigma_w^M \text{ where }\Sigma_w^M=\Sigma_w^G\cap M.
\]

The case where $\m=\u$ is of compact type is well-known. Assume that the triangular decomposition is compatible with $\u$
in the sense that $\t:=\u\cap \h$ is maximal abelian in $\u$. In this event each component $\Sigma_w^U$ is connected and diffeomorphic to a product of an affine space with the compact torus $T=\exp(\t)$.  A reinterpretation of (Bott-Samelson or Soibelman or) Lu coordinates on Schubert cells in $U/T$ (see \cite{Lu} and chapter 5 of \cite{CP}) leads to an explicit real-algebraic parameterization of the affine factor by a product of affine planes via a map which depends on a reduced decomposition of $w$ into a product of simple reflections. One of several remarkable features of this parameterization is that for the open dense component $\Sigma_1^U$, the restriction of Haar measure on $U$ factors as a product of explicit measures on each of planar factors together with the Haar measure on $T$. There is one complex parameter $\zeta_\tau$ for each positive root $\tau$, which we call ``root subgroup coordinates,'' and an element $t\in T$ determining $g\in \Sigma_1^U$.
When all the roots have the same length, in terms of these parameters, the restriction of Haar measure on $U$ to $\Sigma_1^U$ is given by
\begin{equation}\label{compact_Haar_formula}
d\lambda_{U}(g)=d\lambda_T(t)\prod_{\tau>0}\frac{|d\zeta_\tau|}{(1+|\zeta_\tau|^2)^{1+height(\tau)}}
\end{equation}
up to a multiplicative constant.

In this paper, we are mainly interested in the case $\m=\g_0$, a non-compact Lie algebra which is in duality with $\u$, in the
sense of symmetric space theory. Riemannian symmetric spaces come in dual pairs, one of compact
type and one of noncompact type. Given such a pair, there is a
diagram of finite dimensional groups
\begin{equation}\label{fdgroupdiagram}
\xymatrix{ & G & \\ G_0 \ar[ur] &  & U \ar[ul] \\
 & K \ar[ur] \ar[ul] & }
\end{equation}
where $U$ is the universal covering of the identity component of the isometry group
of the compact type symmetric space $X\simeq U/K$, $G$ is the complexification of $U$, and
$G_0\subset G$ is the isometry group for the dual noncompact
symmetric space $X_0=G_0/K$.  The fundamental example is the data determined by the Riemann sphere and the Poincar\'e disk. For this pair, the diagram (\ref{fdgroupdiagram}) becomes
\begin{equation}\label{rankonegroupdiagram}
\xymatrix{ & \SL(2,\mathbb C) & \\
\SU(1,1) \ar[ur] & & \SU(2) \ar[ul] \\
& \mathrm{S}(\U(1)\times \U(1)) \ar[ur] \ar[ul] & }
\end{equation}

Throughout this paper we assume that the Cartan involution $\Theta$ fixing $K$ in $G_0$ is inner; this is equivalent to a number of other conditions:
\begin{list}{$\bullet$}{}
\item $\mathrm{rank}(K)=\mathrm{rank}(G_0)$
\item $G_0$ has discrete
series unitary representations;
\item (if $\g_0$ is simple) $C(K)=S^1$;
\item the
quotients $U/K$ and $G_0/K$ are Hermitian symmetric.
\end{list}
We also assume that the triangular decomposition of $\g=\n^-+\h+\n^+$ is $\Theta$-stable.

With these assumptions, we show that the Birkhoff component $\Sigma_w^{G_0}$ (when nonempty) is connected and diffeomorphic to a product of a contractible bounded complex domain with the torus $T$. We then introduce ``root subgroup coordinates,'' which depends on a reduced decomposition of $w$ into a product of simple reflections.  A key difference between this non-compact inner case and the compact case is that in the non-compact case roots are of two types: compact and non-compact.  Correspondingly, the domain of the parameterization must be altered from a product of affine planes to a product of planes and disks.  As in the compact case, for the open dense component $\Sigma_1^{G_0}$, Haar measure for $G_0$ is equivalent to a product of explicit measures on the planar and disk factors together with Haar measure for $T$.  There is one complex parameter $\zeta_\tau$ for each positive root $\tau$ and an element $t\in T$ determining $g\in \Sigma_1^{G_0}$. When $\tau$ is of compact type, $\zeta_\tau$ is unrestricted in $\C$, while when $\tau$ of non-compact type $\zeta_\tau$ is restricted to the unit disk in $\C$. Assuming all roots are of the same length, in terms of these parameters, the restriction of Haar measure to $\Sigma_1^{G_0}$ is given by
\begin{equation}\label{noncompact_Haar_formula}
d\lambda_{G_0}(g)=d\lambda_T(t)\prod_{\tau>0}\frac{|d\zeta_\tau|}{(1\pm|\zeta_\tau|^2)^{1+height(\tau)}}
\end{equation}
up to a multiplicative constant, where the plus (resp. minus) sign is assigned to $\pm|\zeta_\tau|^2$ if $\tau$ is of compact (reps. non-compact) type.

\subsection{Plan of the Paper}

Section \ref{background} establishes the notation used throughout the paper, consolidating the data used to define Birkhoff factorization and root subgroup factorization into one place in the paper for ease of reference.  Section \ref{factorization_section} concerns Birkhoff decomposition for the groups $U$ and $G_0$ and root subgroup coordinates for the contractible factors of $\Sigma_{w}^U$ and $\Sigma_{w}^{G_0}$. The compact case is relatively well-understood, thanks in large part to Lu (see especially \cite{Lu}). We review and reinterpret this work, with emphasis on the algorithm for root subgroup factorization (which depends on an ordering of non-inverted roots), and then extend it to define root subgroup coordinates for the affine factor of $\Sigma_w^{G_0}$.  The algebra of factorization in the noncompact inner case largely reduces to the compact case, because of the existence of a ``block (or coarse) triangular decomposition." However, there is one part of the argument
in the non-compact case which is not algebraic: this is in showing that everything in a component $\Sigma^{G_0}_w$ has a root subgroup factorization. The paper concludes in Section \ref{Haarmeasuresection} with the computation of the explicit formula (\ref{noncompact_Haar_formula}) for the restriction of Haar measure on $G_0$ to $\Sigma_1^{G_0}$ in terms of root subgroup coordinates.

\section{Notation and Background\label{background}}

Let $ \m$ be a simple Lie algebra over $\R$ and write $\g$ for the complexification $\m^\mathbb C$ of $\m$.  Let $ G$ be the simply connected complex Lie group with Lie algebra $\g$ and let $ M$ denote the connected real subgroup of $ G$ with Lie algebra $\m\subset\g$.  In this section, we will establish notation for studying factorization of elements of $ M$ relative to a Birkhoff decomposition of $ G$. We are mainly interested in the case $\m= \g_0$, a simple noncompact Lie algebra over $\R$, which is
equipped with a Cartan involution $\Theta$ of inner type.

\subsection{Data determined by the choice of a Cartan involution}
The choice of a Cartan involution $\Theta$ on $\g_0$ determines a maximal compact Lie subalgebra $\k$ of $\g_0$.  Let $\sigma$ denote the canonical complex conjugation on $\g$ fixing $\g_0$.  If we extend $\Theta$ to $\g$ in a complex linear fashion then the composition $\tau=\sigma\circ \Theta$ is a complex conjugation on $\g$ fixing a compact real form $\u$ of $\g$.  The extended involution $\Theta$ on $\g$ stabilizes $\u$ and fixes $\k$ inside of $\u$.  Thus, $\g_0\cap\u=\k$.  The assumption that $\g_0$ is of inner type is equivalent to the condition that
\begin{equation}\label{rank_condition}
\mathrm{rank}(\g_0)=\mathrm{rank}({\k})=\mathrm{rank}(\u).
\end{equation}

We write $\g_0=\k+\p$ for the decomposition of $\g_0$ into the eigenspaces of $\Theta$ on $\g_0$.  Then $\u=\k+i\p$ where multiplication by $i$ denotes the canonical complex structure on $\g$, and this is the decomposition of $\u$ into the eigenspaces of the extension of $\Theta$ restricted to $\u$.

Let $ U$ and $ K$ denote the connected subgroups of $ G$ having Lie algebras $ \u$ and $ \k$. Then we obtain a diagram of Lie algebras and corresponding connected Lie groups.
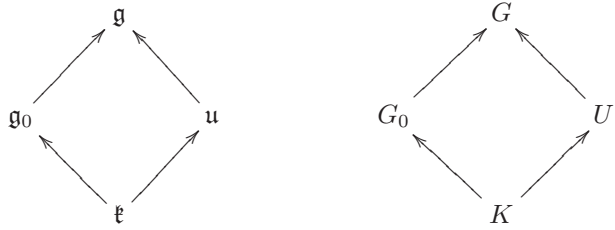
\[
\xymatrix{
 & {\g} & & & & {G} & \\
{\g}_0 \ar[ur] & & {\u} \ar[ul] & & {G}_0 \ar[ur] & & {U} \ar[ul] \\
& {\k} \ar[ur] \ar[ul] & & & & {K} \ar[ur] \ar[ul]&
}
\]
We will also use $\Theta$ to denote the corresponding holomorphic involution of $ G$, and its restrictions to $ G_0$ and $ U$, which fixe $K$ in $ G_0$ and $ U$, respectively.

\subsection{Data determined by the choice of a $\Theta$-stable Cartan subalgebra and a Weyl chamber in the Inner Case}

Fix a Cartan subalgebra ${\t}\subset {\k}$.
Because of our rank assumption (\ref{rank_condition}), ${\t}$ is a
$\Theta$-stable Cartan subalgebra of ${\g}_0$ and every $\Theta$-stable Cartan subalgebra of $\g_0$ is of this form.  In addition, $\t$ is a
$\Theta$-stable Cartan subalgebra of ${\u}$, and its
centralizer ${\h}$ in ${\g}$ is a $\Theta$-stable Cartan
subalgebra of ${\g}$.  We write ${\h}={\t}+\a$, where $\a=i\t$,
for the eigenspace decomposition of $\h$ under $\Theta$ and let $ H=\exp(\h)$, $ T=\exp(\t)$, and $ A=\exp(\a)$, respectively.

We will use $ W:=N_{ U}( T)/ T$ as a model for the Weyl group of $(\g,\h)$.  The choice of a Weyl chamber $C$ in $\a$ determines a choice of positive roots for the action of $\h$ on $\g$. Let $\n^+$ denote the sum of the root spaces indexed by positive roots and $\n^-$ denote the sum of the root spaces indexed by negative roots.  In this way, the choice of a $\Theta$-stable Cartan subalgebra $\t$ of $\g_0$ and a Weyl chamber $C$ determines a triangular decomposition
\begin{equation}\label{triangledecomp}
{\g}={\n}^{-}+ {\h} + {\n}^{+}.
\end{equation}

Set $ N^\pm=\exp(\n^\pm)$.  Then $ B^+= H N^+$ and $ B^-= N^- H$ are a pair of opposite Borel subgroups of $ G$.

A consequence of the stability of $\h$ under $\sigma$ and $\tau$ is the fact that $\sigma(\n^\pm)=\n^\mp$ and $\tau(\n^\pm)=\n^\mp$.

\begin{example}\label{su(1,1)conventions}
In this paper, a special role is played by the rank 1 example of
\[
\mathfrak{su}(1,1)=\left\{\begin{pmatrix}iz & x+iy \\ x-iy & -iz\end{pmatrix}\colon x,y,z\in\R\right\}
\]
with Cartan involution given by
\[
\mathrm{Ad}_g\text{ where }g=\begin{pmatrix} i & 0 \\ 0 & -i\end{pmatrix}.
\]
The complexification is $\mathfrak{sl}(2,\C)$ and the associated compact real form is $\mathfrak{su}(2)$.  The involution fixing $\mathfrak{su}(2)$ is $X\mapsto -X^*$ (opposite conjugate transpose).  The effect of this involution is to negate the off-diagonal entries.  In this case, the maximal compact subalgebra fixed by the involution is the one dimensional subalgebra $\mathfrak{s}(\mathfrak{u}(1)\times \mathfrak{u}(1))$ of diagonal matrices in $\mathfrak{su}(1,1)$. This subalgebra, which is abelian and hence also a Cartan subalgebra, then determines the standard triangular decomposition
\begin{equation}\label{std_tri_decomp_sl(2,C)}
\mathfrak{sl}(2,\C)=\mathrm{span}_\C\left\{\begin{pmatrix} 0& 0 \\ 1 & 0 \end{pmatrix}\right\}+\mathrm{span}_\C\left\{\begin{pmatrix} 1 & 0 \\ 0 & -1 \end{pmatrix}\right\}+\mathrm{span}_\C\left\{\begin{pmatrix} 0 & 1 \\ 0 & 0 \end{pmatrix}\right\}.
\end{equation}
\end{example}

\subsection{Root Data}
Let $\theta$ denote the highest root and normalize the Killing form so that (for the dual form) $\langle \theta,\theta\rangle=2$.  For each root $\alpha$ let $h_{\alpha}\in\a$ denote the associated coroot (satisfying $\alpha(h_{\alpha})=2$).  The inner type assumption, together with the $\Theta$-stability of $\h$, implies that each root space $\g_{\alpha}$ is contained in either $\k^\C$ or in $\p^\C$ and thus the roots can be sorted into two types.  A root $\alpha$ is of \emph{compact type} if the root space $\g_{\alpha}$ is a subset of $\k^\C\subset\g$ and of \emph{noncompact type} otherwise, i.e., when $\g_{\alpha}\subset\p^\C$.
The following is elementary.

\begin{proposition} Suppose $\gamma$ is a positive root, and consider the root homomorphism $\iota_{\gamma}\colon \mathfrak{sl}(2,\C)\to\g_{-\gamma}\oplus\mathbb C h_{\gamma}\oplus\g_{\gamma}$ with $\iota_{\gamma}(diag(1,-1))=h_{\gamma}$ and which carries the standard triangular decomposition of $\mathfrak{sl}(2,\C)$ (\ref{std_tri_decomp_sl(2,C)}) into the triangular decomposition $\g=\n^-+\h+\n^+$. Then
\begin{enumerate}
\item[(a)] $\iota_{\gamma}\colon \mathfrak{su}(2)\to \u$;
\item[(b)] when $\gamma$ is of compact type, $\iota_{\gamma}\colon \mathfrak{su}(2)\to\k$;
\item[(c)] when $\gamma$ is of noncompact type, $\iota_{\gamma}\colon \mathfrak{su}(1,1)\to \g_0$.
\end{enumerate}
\end{proposition}

We denote the corresponding group homomorphism by the same symbol.  Note that if $\gamma$ is of noncompact type, then $\iota_{\gamma}$ induces an embedding of the rank one diagram (\ref{rankonegroupdiagram}) into the group diagram (\ref{fdgroupdiagram}).  For each simple positive root $\gamma$, we use the group homomorphism to set
\begin{equation}\label{defn_of_r_gamma}
\mathbf r_{\gamma}=\iota_{\gamma}\begin{pmatrix} 0 & i \\ i & 0\end{pmatrix}\in N_{ U}( T)
\end{equation}
and obtain a specific representative for the associated simple reflection $r_{\gamma}\in W=N_{ U}( T)/ T$ corresponding to $\gamma$.  (We will adhere to the convention of using boldface letters to denote representatives of Weyl group elements).

\section{Birkhoff and Root Subgroup Factorization\label{factorization_section}}

By definition, the Birkhoff decomposition of $G$ relative to the triangular decomposition $\g=\n^-+\h+\n^+$ is
\begin{equation}
G=\bigsqcup_{W} \Sigma^{G}_{
w} \text{ where }\Sigma^{G}_{ w}=N^- w B^+.
\end{equation}
If we fix a representative $\mathbf w\in N_{U}(T)$ for $w\in W$,
then each $g\in\Sigma^{G}_{w}$ can be factored uniquely as
\begin{equation}\label{birkhoff2}
g=l\mathbf \w ma u,\text{ with } l\in N^-\cap w N^-w^{-1},\ ma\in
TA, \text{ and }u\in N^+.
\end{equation}
This defines functions $l\colon \Sigma_{w}^{G}\to N^-\cap wN^-w^{-1}$, $m\colon \Sigma_{w}^{G}\to T$, $a\colon \Sigma_{w}^{G}\to A$, and $u\colon \Sigma_{w}^{G}\to U$. For fixed $m_0\in T$, the subset $\{g\in \Sigma^{G}_{ w}: m(g)=m_0\}$ is a stratum (topologically an affine space). It is therefore sensible and appropriate to refer to $\Sigma^{G}_{w}$ as the ``isotypic component of the Birkhoff decomposition of $G$ corresponding to $w\in W$."
However we may occasionally lapse into referring to  $\Sigma^{G}_{
w}$ as the ``Birkhoff stratum corresponding to $w$."  We are interested in describing the induced decomposition of $G_0\subset G$.

We say that the elements of $\Sigma_{1}^{G}$ have a \emph{triangular factorization} since then (\ref{birkhoff2}) reduces to
\[
g=l(g)d(g)u(g) \text{ where }d(g)=ma\in TA=H
\]
and $l(g)\in N^-$.  The factor $d(g)$ can be explicitly computed in terms of root data by the formula
\[
d(g)=\prod_{j=1}^r \sigma_j(g)^{h_{\alpha_j}}
\]
where $\sigma_j(g)=\phi_{\Lambda_j}(\pi_{\Lambda_j}(g)v_{\Lambda_j})$ is the fundamental matrix coefficient for the highest weight vector corresponding to $\Lambda_j$.

\subsection{Factorization in the Compact Case}

Although we are interested in the case where $\m=\g_0$ is a non-compact simple Lie algebra of inner type, we record here the corresponding results when $\m=\u$ is a compact simple Lie algebra for comparison.  We will write $U$ for the group $M$ inside of $G$.  Given $w\in W$, define
\[
\Sigma^{U}_{w}:=\Sigma^{G}_{w}\cap U.
\]

\begin{theorem}\label{compact1} Fix a representative $\mathbf w\in N_{U}(T)$ for $w$.
For $g\in\Sigma^{G}_{w}$ the unique factorization (\ref{birkhoff2}) induces a bijective correspondence
\[
\Sigma^{U}_{w}\leftrightarrow \left( N^-\cap \w N^-\w^{-1}\right)\times T\text{ given by }g\mapsto (l,m).
\]
\end{theorem}

\begin{remark} For fixed $m_0\in T$, the set $\{g\in\Sigma^{U}_{w}:m(g)=m_0\}$ is a stratum, and we will refer to $\Sigma^{U}_{w}$ as the ``isotypic component of the Birkhoff decomposition for $U$ corresponding to $w\in W$.'' The quotient of $\Sigma^{U}_{w}$ by $T$ is the usual Birkhoff stratum for the flag space $ U/ T=G/B^+$ corresponding to $w$.
\end{remark}

We now briefly recall Lu's approach to root subgroup factorization from \cite{Lu}. This involves the Bruhat decomposition $G=\bigsqcup_{W} B^+w B^+$.  A translation of Lu's results over to the Birkhoff decomposition will be given below.

For $\zeta\in \C$, we define a function $k\colon \C\to \SU(2)$ by
\begin{equation}\label{uofzeta}
k(\zeta)=\mathbf a_+(\zeta)\left(\begin{matrix}1&-\bar{\zeta}\\\zeta&1\end{matrix}\right)\in SU(2),\quad \text{  where }
\mathbf a_+(\zeta)=(1+|\zeta|^2)^{-1/2}.
\end{equation}
\begin{theorem} Fix $w'\in W$. Choose a minimal factorization $w'=r_{l(w')}\dots r_1$, where each $r_j$ is a reflection and write $\gamma_j$ for the corresponding simple positive root. Then the map
\[
\mathbb C^{l(w')}\times T \to U\cap B^+w'B^{+}\text{ given by }((\zeta_j),t)\mapsto \mathbf r_ni_{\gamma_n}(k(\zeta_n))..\mathbf r_1i_{\gamma_1}(k(\zeta_1))t
\]
is a bijection.
\end{theorem}

\begin{remarks}\label{weylalgorithm}$\phantom{a}$
\begin{enumerate}
\item[(a)] The algorithm for choosing a factorization for $w'$ is the following: choose (a simple positive root) $\gamma_1$ such that $w'\cdot \gamma_1<0$, determining $r_1$; choose $\gamma_2$ such that  $w'r_1\cdot\gamma_2<0$,  determining $r_2$; choose $\gamma_3$ such that $w'r_1r_2\cdot \gamma_3<0$, determining $r_3$; and so on. The positive roots flipped to negative roots are $\tau_j=r_1..r_{j-1}\cdot \gamma_j$, for $j=1,..,l(w')$.
\item[(b)] A choice of factorization $w'=r_{l(w')}..r_1$ determines a
non-repeating sequence of adjacent Weyl chambers
\begin{equation}
C, \ (w_1')^{-1}C,...\ ,\ (w'_j)^{-1}C,...\ ,\ (w')^{-1}C
\end{equation}
where $w_j':=r_j..r_1$ and the
step from $(w_{j-1}')^{-1}C$ to $(w_j')^{-1}C$ is implemented by the
reflection $(w_{j-1}')^{-1}r_jw_{j-1}'$ associated to $\tau_j$.  In particular the
wall between $(w_{j-1}')^{-1}C$ and $(w_j')^{-1}C$ is fixed by $(w_{j-1}')^{-1}r_jw_{j-1}'$.
Conversely, given a sequence of length $l(w')$ of adjacent chambers $C_1,\dots, C_{l(w')}$ from $C_1=C$ to $C_{l(w')}=(w')^{-1}C$ then there is a corresponding minimal factorization.
\item[(c)] A basic example of a factorization of the longest element of the Weyl group for $\sl(n,\C)$ is the lexicographic factorization
\[
r_{{\alpha}_1} (r_{{\alpha}_2} r_{{\alpha}_1}) (r_{{\alpha}_3} r_{{\alpha}_2} r_{{\alpha}_1}).. (r_{{\alpha}_{n-1}}..r_{{\alpha}_1}).
\]
Here $\lambda_i$ denotes the functional which selects the $i^{th}$ diagonal entry, the simple positive roots are ${\alpha}_i=\lambda_{i}-\lambda_{i+1}$, and the sequence of roots $\tau$ is given by
\begin{eqnarray*}
 \tau_1=\lambda_1-\lambda_2,\tau_2=\lambda_1-\lambda_3,\dots,\tau_{n-1}=\lambda_1-\lambda_n, & &\\
 \tau_{n}=\lambda_2-\lambda_3,\dots,\tau_{2n-3}=\lambda_2-\lambda_n,& & \\
 \vdots & & \\
 \tau_{n(n-1)/2}=\lambda_{n-1}-\lambda_n. & &
\end{eqnarray*}
\end{enumerate}
\end{remarks}

The Bruhat and Birkhoff decompositions (in this finite dimensional context) are related by translation by the unique longest Weyl group element $w_0$, i.e.,
\[
N^-wB^+ = w_0 N^+ w'B^+=w_0B^+ w'B^+
\]
where $w=w_0w'$. Since we can choose representatives for $w$, $w_0$, and $w'$ in $U$, the same relationship holds on the induced decompositions of $U$.  In terms of this translation, the following lemma describes how to select the sequence of simple positive roots intrinsically in terms of $w$ (without reference to $w_0$ and $w'$).

\begin{lemma}\label{weylfactorization} Fix $w\in W$.
\begin{enumerate}
\item[(a)] Choose a sequence of simple positive roots $\gamma_j$ in the
following way: (1) choose $\gamma_1$ such that $w\cdot
\gamma_1>0$; (2) choose $\gamma_2$ such that $wr_1\cdot
\gamma_2>0$; (3) choose $\gamma_3$ such that $w r_1r_2\cdot
\gamma_3>0$, and so on, where $r_j$ is the simple reflection
corresponding to $\gamma_j$. Let
$\tau_j=r_1..r_{j-1}\cdot\gamma_j$. Then the $\tau_j$ are the
positive roots which are mapped to positive roots by
$w$.
\item[(b)] This choice of positive roots determines a
non-repeating sequence of adjacent Weyl chambers
\begin{equation}
w^{-1}C,\ r_1w^{-1}C,\ ...,\ r_{j-1}..r_1w^{-1}C,\ ...,\ -C.
\end{equation}
If $w_j'=r_j..r_1$ then the step from $w_{j-1}'w^{-1}C$ to $w_j'w^{-1}C$ is implemented by the reflection $r_j$.  Conversely, given a sequence $(C_j)$ of length $n=l(w_0)-l(w)$ consisting of adjacent chambers from
$C_1=w^{-1}C$ to $C_n=-C$ which is minimal, there is a corresponding minimal factorization of $w'=w_0^{-1}w$.
\end{enumerate}
\end{lemma}

\begin{proof} As we noted above, this is equivalent to the more
standard procedure of setting $w'=w_0^{-1}w$ and choosing a
reduced factorization $w'=r_n..r_1$ where $n=l(w_0)-l(w)=l(w')$.
\end{proof}

\begin{theorem}\label{Lutheorem} Fix $w\in W(K)$ and a representative $\mathbf w\in N_{K}(T)$ for $w$, then determine positive simple roots $\gamma_1,\dots,\gamma_n$ with associated simple reflections $r_1,\dots,r_n$, and positive roots $\tau_1,\dots,\tau_n$ as in Lemma \ref{weylfactorization}.  Set $\w_j'=\mathbf r_j..\mathbf r_1$ and $\iota_{\tau_j}(g)=\w_{j-1}'\iota_{\gamma_j}(g)(\w_{j-1}')^{-1}$ for each $g\in \SL(2,\C)$.
\begin{enumerate}
\item[(a)] Each $g\in \Sigma^{U}_w$ has a unique factorization of the form
\[
g=\mathbf w \iota_{\tau_n}(k(\zeta_n))..\iota_{\tau_1}(k(\zeta_1)) t
\]
for some $t\in T$ and some $(\zeta_1,\dots,\zeta_n)\in \C^n$.
\item[(b)] The map
\[
\C^n\to N^{-}\cap wN^{-}w^{-1}:\zeta \to l(\mathbf w \iota_{\tau_n}(k(\zeta_n))..\iota_{\tau_1}(k(\zeta_1)))
\]
is a diffeomorphism.
\item[(c)] If $g\in \Sigma_{w}^{\dot{U}}$ has the factorization in part (a) then the factor $a(g)$ from (\ref{birkhoff2}) has the product form
    \[
    a(g)=\prod_{j=1}^n
\mathbf a_+(\zeta_j)^{h_{\tau_j}}
\]
where $\mathbf a_+$ is the function from (\ref{uofzeta}).
\end{enumerate}
\end{theorem}

\begin{proof} This is a translation of Lu's results. We will essentially reproduce the proof in the
next subsection.
\end{proof}

\subsection{Factorization in the Noncompact Inner Case}

Now we return to the case where $\m=\g_0$ is a noncompact Lie algebra over $\R$ of inner type.  Then each root space for $\h$ on $\g$ is contained either in $\k^\C$ or in $\p^\C$.  This yields a vector space decomposition
\[
\n^+=\n^+_{\k}+\n^+_{\p}
\]
where $\n^+_{\k}$ is spanned by root vectors corresponding
to compact type positive roots, and $\n^+_{\p}$ is spanned by root
vectors corresponding to noncompact type positive roots; moreover,
$\n^+_{\p}$ is an abelian ideal of $\n^+$.  Likewise, $\n^-=\n_\p^-+\n_{\k}^-$.
Note that
\begin{equation}\label{cartandecomp}
\k^{\mathbb C}=\n_\k^-+\h+\n_\k^+
\text{ and }\mathfrak p^{\mathbb C}=\n_{\p}^-+\n_{\p}^+
\end{equation}
as vector spaces.  The sum $\k^\C+\n_\p^+$ is the parabolic subalgebra of $\g$ corresponding to the set of simple positive roots of compact type.  We denote the corresponding parabolic subgroup of $G$ by $P=K^\C N_\p^+$.

We refer to the decomposition
\[
\mathfrak g=\mathfrak n^{-}_\p+ {\mathfrak k}^{\mathbb C}+{\mathfrak
n}^{+}_\p
\]
as a \emph{block triangular decomposition} of $\g$.  For $G$, the corresponding group-level ``block Birkhoff decomposition'' is well-known.  We will use $W(K)=N_{K}(T)/T$ as a model for the Weyl group of $(\k,\t)$.  Then there is a faithful embedding
\[
W(K):=N_{K}(T)/T\to W:=N_{U}(T)/T.
\]
The components of the ``block Birkhoff decomposition'' are then indexed by the elements of the quotient $W/W(K)$.

\begin{lemma}\label{first lemma}
\[
G=\bigsqcup_{w\in W/W(K)}N^-_pwP
\]
and the natural map
\[
N^-_p\cap wN^-_pw^{-1}\to N^-_p wP\text{ given by }l\mapsto lwP
\]
is a diffeomorphism.
\end{lemma}

\begin{proof} This follows from standard facts about Birkhoff
stratification for a generalized flag space; see, for example, the Appendix to \cite{Pi1}.
\end{proof}

\begin{theorem}\label{first theorem} $\phantom{a}$
\begin{enumerate}
\item[(a)] Each $g\in G_0$ has a unique ``block triangular factorization"
\[
g=l_p g_k u_p\text{ where }l_p\in N^-_\p,\ g_k\in K^{\mathbb C},\ \text{ and }u_p\in N^+_\p.
\]
\item[(b)] The set $D(G_0,K)=\{l_p:g=l_pg_ku_p\in G_0\}$ is a contractible bounded complex domain in $N^-_p$.
\item[(c)] Each $g\in G_0$ has a factorization of the form
\begin{equation}\label{birkhoff_refined}
g=l_p l_k ma \mathbf w u_k u_p
\end{equation}
where $l_p$ and $u_p$ are the same factors occurring in part (a),
$\mathbf w$ is a representative for some $w\in W(K)$, $ma\in TA$, $l_k \in N_k^-\cap
w N_k^-w^{-1}$, and $u_k\in N_k^+$.
\item[(d)] Furthermore,
\[
G_0=\bigsqcup_{W(K)} \Sigma^{G_0}_{w}, \text{ where }\Sigma^{G_0}_{ w}:=\Sigma^{G}_{w}\cap G_0.
\]
\item[(e)] For each $w\in W(K)$, if a representative $\mathbf w$ for $w$ is fixed then the factorization in part (c) is unique and defines functions $l_p\colon \Sigma_{w}^{G_0}\to D(G_0,K)$, $l_k\colon \Sigma_{w}^{G_0}\to N_k^-\cap wN_k^- w^{-1}$, $m\colon \Sigma_{w}^{G_0}\to T$, and  induces a diffeomorphism
\[
\Sigma^{G_0}_{w}\to D(G_0,K)\times
\{l_k\in N_k^-\cap wN_k^-w^{-1}\}\times T
\]
given by $g\mapsto (l_p(g),l_k(g),m(g))$.
\end{enumerate}
\end{theorem}

\begin{proof} There is a natural map $G_0/K\to G/P$ and by Lemma \ref{first lemma}, $G/P$ is a union of $N_p^-$-orbits indexed by $w\in W/W(K)$.  The top stratum corresponds to $w=1$ in $W/W(K)$ and is parameterized by $N_p^-$.  Theorem 5 of \cite{Pi0} shows that this image is contained in the top stratum.  This can also be deduced from Lemma 7.9 on page 388 of \cite{H}.  Parts (a) and (b) of the theorem now follow from this observation.  This also implies that $\Sigma_{w}^{G_0}$ is empty unless $w\in W(K)$ (since $W(K)$ is the $1$ in $W/W(K)$).  Hence part (d) follows from this observation as well.

Part (c) is a consequence of the triangular factorization for $K^{\mathbb C}$ (this is the sense in which
triangular factorization for the noncompact inner case reduces to the compact case).

Because of the uniqueness of the factorization in (c), the map in (e) is 1-1. It remains to show the map is onto. Suppose that we are given $(l_p,l_k,t)$ in the codomain of the map. By definition, there
exists $g_0\in G_0$ such that $g_0=l_pg_ku_p$. Given $k\in K$,
\[
g_0k=l_pg_ku_pk=l_pg_k'u_p',\text{ where } g_k'=g_kk\in K^{\mathbb C} \text{ and }u_p'=k^{-1}u_pk\in N_p^+
\]
because $K$ normalizes  $N_p^+$. Since $K$ acts transitively on the flag space  $(B^-\cap K^{\mathbb C})\backslash K^{\mathbb C}$, we can choose $k$ such that $g_k'$ has triangular factorization of the form
$l_k\mathbf w mau_k$. We can always multiply on the right by a $t'$ to obtain the desired $t$ without affecting the other factors.  This shows the map is onto.
\end{proof}

\begin{corollary}\label{topology} $\phantom{a}$
\begin{enumerate}
\item[(a)] The map $G_0 \to K^{\mathbb C}$ given by $g\to g_k$ is a homotopy equivalence.
\item[(b)] For each $w\in W(K)$ the map
\[
\Sigma^{G_0}_{w}\to \Sigma^{K^{\mathbb C}}_{w}\text{ given by }g \mapsto g_k
\]
induced by part (c) of Theorem \ref{first theorem} is a homotopy equivalence.
\end{enumerate}
\end{corollary}

We now turn to root subgroup factorization. In the compact case, a special role was played by the function $k\colon \C\to\mathrm{SU}(2)$ defined by
\begin{equation}\label{factor_plus}
k(\zeta)=\mathbf a_+(\zeta)
\begin{pmatrix}1&-\bar{\zeta}\\\zeta&1
\end{pmatrix}=\begin{pmatrix}1&0\\\zeta&1
\end{pmatrix}\begin{pmatrix}\mathbf a_+(\zeta)&0\\0&\mathbf a_+(\zeta)^{-1}
\end{pmatrix}\begin{pmatrix}1&-\bar{\zeta}\\0&1
\end{pmatrix} \in \SU(2)
\end{equation}
where $\mathbf a_+(\zeta)=(1+\vert\zeta\vert^2)^{-1/2}$.  By composing copies of this function with root homomorphisms, interleaving the compositions into a minimal sequence of simple reflections factoring $w$, and multiplying out the results we were able to parameterize $\Sigma_w^{U}$.  To accommodate the new situation where the simple reflections may be associated to noncompact type roots, and to parameterize $\Sigma_{w}^{\dot{G}_0}$, we define a function $q\colon \Delta\to \mathrm{SU}(1,1)$ by
\begin{equation}\label{factor_minus}
q(\zeta)=\mathbf a_-(\zeta)
\begin{pmatrix}1&\bar{\zeta}\\\zeta&1
\end{pmatrix}=\begin{pmatrix}1&0\\\zeta&1
\end{pmatrix}\begin{pmatrix}\mathbf a_-(\zeta)&0\\0&\mathbf a_-(\zeta)^{-1}
\end{pmatrix}\begin{pmatrix}1&\bar{\zeta}\\0&1
\end{pmatrix} \in \SU(1,1)
\end{equation}
where $\mathbf a_-(\zeta)=(1-\vert\zeta\vert^2)^{-1/2}$.

\begin{theorem}\label{root_factorization_g0} Fix $w\in W(K)$ and a representative $\mathbf w\in N_{ K}(T)$ for $w$, then determine positive simple roots $\gamma_1,\dots,\gamma_n$ with associated simple reflections $r_1,\dots,r_n$, and positive roots $\tau_1,\dots,\tau_n$ as in Lemma \ref{weylfactorization}.  Set $\w_j'=\mathbf r_j..\mathbf r_1$ and $\iota_{\tau_j}(g)=\w_{j-1}'\iota_{\gamma_j}(g)(\w_{j-1}')^{-1}$ for each $g\in \SL(2,\C)$.
\begin{enumerate}
\item[(a)] Each $g\in \Sigma^{G_0}_w$ has a unique factorization
\[
g=\mathbf w \,\iota_{\tau_n}(g(\zeta_n))..\iota_{\tau_1}(g(\zeta_1)) t
\]
for some $t\in T$ and $(\zeta_1,\dots,\zeta_n)\in\C^n$, where if $\tau_j$ is a noncompact type
root, then $\vert \zeta_j\vert<1$ and $g(\zeta_j)=q(\zeta_j)$ from (\ref{factor_minus}), and if
$\tau_j$ is a compact type
root, then $\zeta_j$ is unrestricted in $\mathbb C$ and $g(\zeta_j)=k(\zeta_j)$ as in (\ref{factor_plus}).
\item[(b)] If $g\in \Sigma_w^{G_0}$ has the factorization in part (a) then the factor $a(g)$ of $g$ from part (c) of Theorem \ref{first theorem} has the product form
    \[
    a(g)=\prod_{j=1}^n \mathbf a(\zeta_j)^{h_{\tau_j}}
    \]
    where $\mathbf a(\zeta_j)=\mathbf a_-(\zeta_j)$ if $\tau_j$ is a noncompact type root and $\mathbf a(\zeta_j)=\mathbf a_+(\zeta_j)$ if $\tau_j$ is a compact type root.
\end{enumerate}
\end{theorem}

\begin{proof}We must show that the map
\begin{equation}\label{openmap}
\{(\zeta_1,..,\zeta_n)\} \times T\to \Sigma^{G_0}_w\text{ given by }((\zeta_j),t) \mapsto g
\end{equation}
where $g$ is defined as in part (a), is a
diffeomorphism. Here it is understood that if the $j$th root is of
noncompact type, then $g(\zeta_j)=q(\zeta_j)$ and $\vert \zeta_j\vert<1$, and if the $j$th root is of
compact type, then $g(\zeta_j)=k(\zeta_j)$ and $\zeta_j$ is unrestricted in $\mathbb C$.

We first calculate the triangular decomposition for
\[
g^{(n)}:=\iota_{\tau_n}(g(\zeta_n))..\iota_{\tau_1}(g(\zeta_1))
\]
by induction on $n$. In the process we
will prove part (b), which will be used in the proof of part (a).  First note that since $\tau_j=(w_{j-1}')^{-1}\cdot \gamma_j$ and $\iota_{\tau_j}$ preserves triangular factorizations,
\begin{eqnarray*}
\iota_{\tau_j}(g(\zeta_j)) & = & \iota_{\tau_j}(
\begin{pmatrix}1&0\\\zeta_j&1
\end{pmatrix})\mathbf a_\pm(\zeta_j)^{h_{\tau_j}}\iota_{\tau_j}(
\begin{pmatrix}1&\pm\bar{\zeta_j}\\0&1
\end{pmatrix}) \\
& = & \exp(\zeta_j
f_{\tau_j})\mathbf a_\pm(\zeta_j)^{h_{\tau_j}}(\mathbf
w_{j-1}')^{-1}\exp(\pm\bar{\zeta}_j e_{\gamma_j})\mathbf w_{j-1}'
\end{eqnarray*}
is a triangular factorization (the plus/minus case is used for the compact/noncompact root case, respectively).   In what follows, we will simply write $\mathbf a(\zeta_j)$ for $\mathbf a_\pm(\zeta_j)$ since the appropriate sign can be inferred from the type of the corresponding root $\tau_j$.

Suppose that $n=2$. Then
\begin{equation}\label{n=2case1}
g^{(2)}=\exp(\zeta_2 f_{\tau_2})\mathbf a(\zeta_2)^{h_{\tau_2}}\mathbf
r_1^{-1}\exp(\pm\bar{\zeta}_2 e_{\gamma_2})\mathbf r_1 \exp(\zeta_1
f_{\gamma_1})\mathbf a(\zeta_1)^{h_{\gamma_1}}\exp(\pm\bar{\zeta}_1
e_{\gamma_1})
\end{equation}
where $\pm \bar{\zeta}_j$ occurs according to whether $\tau_j$ is a compact/noncompact type root, respectively.  The key point is that
\begin{eqnarray*}
&  & \mathbf r_1^{-1}\exp(\pm\bar{\zeta}_2 e_{\gamma_2})\mathbf r_1 \exp(\zeta_1 f_{\gamma_1}) \\
 & = & \mathbf r_1^{-1} \exp(\pm\bar{\zeta}_2 e_{\gamma_2})
\exp(\zeta_1 e_{\gamma_1})\mathbf r_1 \\
& = & \mathbf r_1^{-1} \exp(\zeta_1
e_{\gamma_1})\widetilde u \mathbf r_1,\quad (\text{for some }
 \widetilde u\in N^+\cap r_1^{-1} N^+ r_1) \\
& = & \exp(\zeta_1
f_{\gamma_1})\mathbf u, \quad (\text{for some }  \mathbf u\in
N^+).
\end{eqnarray*}
Insert this calculation into (\ref{n=2case1}). We then see
that $g^{(2)}$ has a triangular factorization $g^{(2)}=l^{(2)}a^{(2)}u^{(2)}$, where
\[
a^{(2)}=\mathbf a(\zeta_1)^{h_{\tau_1}}\mathbf a(\zeta_2)^{h_{\tau_2}}
\]
and
\begin{eqnarray}\label{2ndcase1}
l^{(2)} & = & \exp(\zeta_2f_{\tau_2}) \exp(\zeta_1\mathbf a(\zeta_2)^{-\tau_1(h_{\tau_2})}f_{\tau_1}) \\
& = & \exp(\zeta_2
f_{\tau_2}+\zeta_1\mathbf
a(\zeta_2)^{-\tau_1(h_{\tau_2})}f_{\tau_1}). \nonumber
\end{eqnarray}
Note that with $l$-factor, the sign difference between the calculations in the compact and noncompact case is only present in the corresponding type of $\mathbf a(\zeta_2)$.

To apply induction, we assume that $g^{(n-1)}$ has a triangular
factorization $g^{(n-1)}=l^{(n-1)}a^{(n-1)}u^{(n-1)}$ with
\begin{equation}\label{induction1}
l^{(n-1)}=\exp(\zeta_{n-1}f_{\tau_{n-1}})\widetilde l \in N^-\cap
(w_{n-1}')^{-1}N^+w_{n-1}'=\exp(\sum_{j=1}^{n-1}\mathbb C
f_{\tau_j}),
\end{equation}
for some $\widetilde l \in N^-\cap (w_{n-2}')^{-1}N^+w_{n-2}'=\exp(\sum_{j=1}^{n-2}\mathbb C f_{\tau_j})$,
and
\begin{equation}\label{prod_form_of_a}
a^{(n-1)}= \prod_{j=1}^{n-1}\mathbf a(\zeta_j)^{h_{\tau_j}}.
\end{equation}
We have established this for $n-1=1,2$. For $n \ge 3$
\begin{eqnarray*}
 g^{(n)} & = & \exp(\zeta_n f_{\tau_n})\mathbf a(\zeta_n)^{h_{\tau_n}}(\mathbf w_{n-1}')^{-1}\exp(\pm\bar{\zeta}_n e_{\gamma_n})\mathbf w_{n-1}'\exp(\zeta_{n-1}
f_{\tau_{n-1}})\widetilde l a^{(n-1)}u^{(n-1)} \\
 & = & \exp(\zeta_n f_{\tau_n})\mathbf a(\zeta_n)^{h_{\tau_n}}(\mathbf w_{n-1}')^{-1}\exp(\pm\bar{\zeta}_ne_{\gamma_n}) \widetilde u\mathbf w_{n-1}'a^{(n-1)}u^{(n-1)},
\end{eqnarray*}
where $\widetilde u= \mathbf w_{n-1}'\exp(\zeta_{n-1} f_{\tau_{n-1}})\widetilde l (\mathbf w_{n-1}')^{-1}\in w_{n-1}' N^{-} (w_{n-1}')^{-1} \cap
N^+$. Now we factor $\exp(\pm\bar{\zeta}_ne_{\gamma_n}) \widetilde
u\in N^+$ as $\widetilde u_1\widetilde u_2$, relative to the decomposition
\[
N^+=\left(N^+\cap w_{n-1}'N^-(w_{n-1}')^{-1}\right)\left(N^+\cap
w_{n-1}'N^+(w_{n-1}')^{-1}\right)
\]
and let
\[
\mathbf l=\mathbf a(\zeta_n)^{h_{\tau_n}}(\mathbf
w_{n-1}')^{-1}\widetilde u_1 \mathbf w_{n-1}'\mathbf
a(\zeta_n)^{-h_{\tau_n}}\in N^- \cap
(w_{n-1}')^{-1} N^+
w_{n-1}.
\]
Then $g^{(n)}$ has triangular decomposition
\begin{eqnarray*}
g^{(n)} & = & \left(\exp(\zeta_n f_{\tau_n})\mathbf
l\right) \left(\mathbf
a(\zeta_n)^{h_{\tau_n}}a^{(n-1)}\right)\left(
(a^{(n-1)})^{-1}\mathbf w_{n-1}'\widetilde u_2(\mathbf w_{n-1}')^{-1}a^{(n-1)}u^{(n-1)}\right) \\
& = & l^{(n)}a^{(n)}u^{(n)}
\end{eqnarray*}
by (\ref{prod_form_of_a}).

Now suppose that we multiply this triangular decomposition on the left by $\mathbf w$ (as in part (a)).
Because the $\tau_j$, $j=1,..,n$, are the positive roots which are mapped
to positive roots by $w$, it follows that $l^{(n)}$ will be conjugated by $\mathbf w$ into
another element in $N^-$.  It follows that $g$, as defined in part (a), is in $\Sigma^{G_0}_w$.

Now we want to draw some conclusions. First note that the inductive calculation of the triangular decomposition
implies part (b) of the Theorem. We can also see that
the map (\ref{openmap}), which has domain a product of disks and affine planes, is 1-1 and open. Because  $\Sigma^{G_0}_w$ is connected (by
Corollary \ref{topology}), to conclude that the map (\ref{openmap}) is a diffeomorphism, it suffices to show
that the map (\ref{openmap}) has a closed image in $\Sigma^{G_0}_w$. Suppose that $(g_k)_{k=1}^\infty$ is a sequence of elements
in the image of (\ref{openmap}) (these are elements that have root subgroup factorizations), and suppose this sequence converges to $g'\in\Sigma^{G_0}_w$. We must show that $g'$ has a root subgroup factorization. For each $k$, consider the unique triangular factorization
\[
g_k=l_k \mathbf wm_ka_k u_k, \quad  l_k\in N^-\cap w N^-w^{-1}.
\]
Since $g'\in\Sigma^{G_0}_w$ with triangular factorization $g'=l'\mathbf wm'a'u'$, $l'\in N^-\cap w N^- w^{-1}$, we know that $a_k\to a'$ as $k\to\infty$, and hence the sequence $(a_k)_{k=1}^\infty$ is bounded in $A$.

We now will show that the formula
in part (b) for the $a$-component, applied to each $g_k$ (together with a pivotal fact about roots for Hermitian symmetric spaces), implies that the associated sequence of parameters in the domain must remain bounded as $k\to\infty$. This formula can be written as
$$ a(g_k)=\prod_{j=1}^n \mathbf a(\zeta^{(k)}_j)^{h_{\tau_j}} =\prod_{i=1}^{r}\left(\prod_{j=1}^n(1\pm\vert\zeta^{(k)}_{\tau_j}\vert^2)^{-\frac12 m_i(\tau_j)h_{\alpha_i}}\right)^{h_{\alpha_i}}  $$
where the $\alpha_i$, $1\le i\le r$, are the simple positive roots, for a root $\gamma$ we write $\gamma=\sum m_i(\gamma) \alpha_i$, and the $\pm$ is positive if and only if the corresponding root is of compact type. What we can a priori conclude from this formula is that for each $1\le i\le rank(\g)$,
\begin{equation}\label{keyprod}\prod_{j=1}^n(1\pm\vert\zeta^{(k)}_{\tau_j}\vert^2)^{-m_i(\tau_j)}\end{equation}
remains bounded as $k\to\infty$. If all of the signs are the same, then we can conclude that the sequence $\zeta^{(k)}_{\tau_j}$
remains bounded (in the plane for a compact type root, and in the disk for a noncompact type root) as $k\to\infty$; however, if there are mixed signs, then there could be terms going to zero and terms going to infinity which perfectly balance, and we cannot draw any conclusion.

For simplicity we now assume $\g_0$ is simple (otherwise we work with its unique simple components). The basic fact about
roots of an irreducible Hermitian symmetric space is that there is exactly one simple root, say $\phi=\alpha_r$, such that each positive root
can be written either as
$$\alpha=\sum_{l=1}^{r-1}m_l(\alpha)\alpha_l \text{ or } \alpha=\phi+\sum_{l=1}^{r-1}m_l(\alpha)\alpha_l$$
In the first case $\alpha$ is of compact type, and in the second case $\alpha $ is of non-compact type.
This follows from Lemma 3 of \cite{Wolf} (Remark: this fact is the key to Wolf's proof of Cartan's classification of Hermitian
symmetric spaces; of course one could also use the classification of such spaces to deduce this as well).
For $i=r$ in the preceding paragraph, all of the signs in the product \ref{keyprod} are negative, and all the $\tau_j$ which
are of noncompact type occur in this product. Thus for all of the $\tau_j$ which are of noncompact type,
the sequence $\zeta^{(k)}_{\tau_j}$ will be bounded in the disk as $k\to\infty$. Once this is established, we can consider $i<r$
and draw the same conclusion for $\tau_j$ of compact type.

We can therefore find a subsequence of the sequence of parameters which converges to an element of the domain (the plane for
a compact type root, the disk for a noncompact type root). The sequence $g_k$ will then converge to the group element
corresponding to this limiting parameter by continuity.  This limit must be $g'$, and hence we obtain a root subgroup factorization for $g'$. This completes the proof.
\end{proof}

\section{Haar Measure in Root Subgroup Coordinates}\label{Haarmeasuresection}

In closing we mention one striking feature of root subgroup factorization, the fact that Haar measure
is a product in these coordinates. The analogue of this in the compact case is due to Lu in \cite{Lu}, where she obtains product formulas for Kostant's harmonic forms on $\dot{U}/\dot{T}$, one of which is the invariant volume on $\dot{U}/\dot{T}$.  The argument here is more direct.

\begin{theorem}\label{haartheorem} With $w=1$ and notation as in Theorem \ref{root_factorization_g0} we obtain a parametrization of the open dense subset $\Sigma_1^{\dot G_0}\subset \dot G_0$ by a product of copies of complex planes and disks, together with the torus $T$. In terms of the complex parameters $(\zeta_j)$ and $t\in \dot T$ for
\[
g= i_{\tau_{\mathbf n}}(g(\zeta_{\mathbf n}))..i_{\tau_1}(g(\zeta_1)) t \in \Sigma^{\dot G_0}_1
\]
where $g(\zeta_j)=k(\zeta_j)$ when $\tau_j$ is of compact type (resp. $g(\zeta_j)=q(\zeta_j)$ when $\tau_j$ is of noncompact type), then Haar measure for $\dot{G}_0$ is (up to a constant)
\[
d\lambda_{\dot G_0}(g) =\left(\prod_{j=1}^{\mathbf n} \mathbf a(\zeta_{j})^{2\delta(h_{\tau_j})+2}\vert d\zeta_j \vert\right) d\lambda_{\dot T}(t)
\]
where $\mathbf a(\zeta_j)=\mathbf a_\pm(\zeta_j)$ according to whether $\tau_j$ is of compact/noncompact type, and where $d\lambda_{\dot T}(t)$ denotes Haar measure for $\dot{T}$.
\end{theorem}

\begin{remark}\label{haarremark} (a) This formula for Haar measure can alternatively be
written as
\[
d\lambda_{\dot G_0}(g)=d\lambda_{\dot T}(t)\prod_{\tau>0}\frac{\vert d\zeta_{\tau}\vert}{(1\pm\vert \zeta_{\tau}\vert^2)^{1+\delta(h_{\tau})}}
\]
where we choose the plus sign for compact roots, the negative sign for noncompact roots, $\zeta_{\tau}$ is understood
to be bounded by one when $\tau$ is noncompact, and $\zeta_{\tau}=\zeta_j$ when $\tau=\tau_j$.

(b) When $\dot{\mathfrak g}$ is simply laced, $\delta(h_{\tau})$ is the height of the positive root $\tau$. Thus in the simply laced case,
\[
d\lambda_{\dot G_0}(g)=d\lambda_{\dot T}(t)\prod_{\tau>0}\frac{\vert d\zeta_{\tau}\vert}{(1\pm\vert \zeta_{\tau}\vert^2)^{1+height(\tau)}}\]
\end{remark}

Denote the triangular decomposition for
$g\in \Sigma^{\dot G_0}_1$ by
$$g=l(g)m(g)a(g)u(g)$$ Recall that $g$ is uniquely determined by $l(g)\in\dot N^-$ and $m(g)\dot T$.  The following formula should be attributed to Harish-Chandra:

\begin{lemma}\label{haarlemma} Up to a normalization
\[
d\lambda_{\dot G_0}(g)=a(g)^{4\delta}d\lambda_{\dot N^-}(l(g)) d\lambda_{\dot T}(m(g))
\]
where (by slight abuse of notation) it is understood that we are restricting Haar measure for $\dot N^-$ to the intersection of $\dot N^-$ with the image of $\Sigma^{\dot G_0}_1$.
\end{lemma}

\begin{proof}
This is equivalent to proving the coordinate expression
$$d\lambda_{\dot G_0/\dot T}(g \dot T)=a(g)^{4\dot{\delta}}d\lambda_{\dot N^-}(l(g))$$
for the invariant measure on the quotient $\dot G_0/\dot T$. The value of the density of $d\lambda_{\dot G_0/\dot T}$ with respect to $d\lambda_{\dot N^-}(l(g))$ at $l(g)$ can be computed as follows.  Identify the tangent space to $\dot N^-$ at $l(g)$ with $\n^-$ by left translation.  The derivative at $1\in l(\Sigma_1^{\dot{G}_0})\subset N^-$ of left translation by $g\in \dot{G}_0$ is then identified with a linear map from $\n^-$ (viewed as the tangent space to $\dot N^-$ at $1$) to $\n^-$ (viewed as the tangent space at $l(g)$).  The reciprocal of the determinant of this map is the value of the density at $l(g)$.

Given $X\in\n^-$, the curve $\varepsilon\mapsto \exp{\varepsilon X}$ represents the corresponding tangent vector at $1\in \dot N^-$.  Let $\varepsilon\mapsto g_0(\varepsilon)$ denote a lift of this curve to $\dot{G}_0$, i.e., $l(g_0(\varepsilon))=\exp{\varepsilon X}$.  We can arrange for this lift to have $m(g_0(\varepsilon))=1$ for $\varepsilon$ small.  Then
\begin{equation}\label{t_vector}
\varepsilon\mapsto l(g)^{-1}l(gg_0(\varepsilon))
\end{equation}
represents the image of $\varepsilon\mapsto \exp(\varepsilon X)$ under left translation by $g$ through these identifications.  Let $g_0(\varepsilon)=l(\varepsilon)a(\varepsilon)u(\varepsilon)$ denote the triangular factorization of $g_0(\varepsilon)$.  Then
\begin{eqnarray*}
l(g)^{-1}l(gg_0(\varepsilon)) & = & l(g)^{-1}l(l(g)a(g)u(g)l(\varepsilon)a(\varepsilon)u(\varepsilon))\\
& = & l(\mathrm{Ad}_{a(g)u(g)}(l(\epsilon))a(g)u(g)a(\varepsilon)u(\varepsilon)) \\
& = & l(\exp(\varepsilon \mathrm{Ad}_{a(g)u(g)}(X)))
\end{eqnarray*}
so the derivative of (\ref{t_vector}) at $\epsilon=0$ is the linear map
\begin{equation}\label{derivative_map}
X\mapsto (\mathrm{Ad}_{a(g)u(g)}(X))_-
\end{equation}
where $(\cdot)_-$ denotes the projection to $\n^-$ along the triangular decomposition $\g=\n^-+\h+\n^+$.
We claim that the matrix representing (\ref{derivative_map}) in terms of the basis of negative roots is triangular.  Indeed, if $X\in \n^-$ is homogeneous of a given height then $\mathrm{Ad}_{u(g)}(X)=X+X'(g)$ where $X'(g)$ is a sum of terms of strictly greater height than that of $X$ because $u(g)\in N^+$ and thus $\mathrm{Ad}_{u(g)}$ is unipotent.  Therefore, $\mathrm{Ad}_{a(g)u(g)}(X)=\mathrm{Ad}_{a(g)}(X)+X''(g)$ where again $X''(g)$ is a sum of terms of height strictly greater than $\mathrm{height}(X)=\mathrm{height}(\mathrm{Ad}_{a(g)}(X))$ since $a(g)\in A$.  Thus, the determinant of (\ref{derivative_map}) as a real linear transformation is $a(g)$ raised to twice the sum of the negative roots, i.e., $a(g)^{-4\dot{\delta}}$.  Taking the reciprocal gives the desired formula for the density.
\end{proof}

\begin{lemma}\label{nillemma} In the $\zeta$ coordinates
$$d\lambda_{\dot N^-}(l(g))=\prod_{k=1}^{\mathbf n}\frac{\vert d\zeta_k\vert}{(1\pm\vert \zeta_k\vert^2)^{1-\delta(h_{\tau_k})}}$$
where a sign is positive if and only if the corresponding root is of compact type.
\end{lemma}

\begin{remark} With the same conventions as in Remark \ref{haarremark}, the formula in Lemma \ref{nillemma} can be alternatively written as
\[
d\lambda_{\dot N^-}(l(g))=\prod_{\tau>0}(1\pm | \zeta_{\tau}|^2)^{\delta(h_{\tau})-1}| d\zeta_{\tau}|
\]
In a similar way
\[
a^{4\delta}=\prod_{\tau>0}(1\pm\vert \zeta_{\tau}\vert^2)^{-2\delta(h_{\tau})}.
\]
Together with Lemma \ref{haarlemma}, these formulas immediately imply Theorem \ref{haartheorem} (as formulated in Remark \ref{haarremark}).
\end{remark}

The basic idea of the proof of Lemma \ref{nillemma} is the following. If we write $l(g)=\exp(\sum x_j f_{\tau_j})$, then there is a triangular
relationship between the $\zeta $ variables and the $x$ variables which is implicit in the proof of Theorem \ref{root_factorization_g0}. To carefully prove this, we need to go back through the induction argument in that proof. This
involves an algebraic lemma and a more technical
statement.

\begin{lemma}\label{alglemma} For $n=1,..,\mathbf n$,
$$\delta(h_{\tau_n})-1=\sum_{k=1}^{n-1} \tau_k(h_{\tau_n})$$
\end{lemma}

\begin{proof} Recall $w_{n-1}=r_{n-1}..r_1$. Since $\tau_n=w_{n-1}^{-1}\cdot \gamma_n$, and $Inv(w_{n-1})=\{\tau_k:k<n\}$,
the lemma is equivalent to
$$\frac{1}{2}\left(\sum_{\beta\in Inv(w_{n-1})}w_{n-1}\cdot\beta+\sum_{\beta\notin Inv(w_{n-1})}w_{n-1}\cdot\beta\right)(h_{\gamma_n})-1=\left(\sum_{\beta\in Inv(w_{n-1})}w_{n-1}\cdot\beta\right)(h_{\gamma_n})$$
or
$$\frac{1}{2}\left(\sum_{\beta\in Inv(w_{n-1})}w_{n-1}\cdot(-\beta)+\sum_{\beta\notin Inv(w_{n-1})}w_{n-1}\cdot\beta\right)(h_{\gamma_n})=1$$
The left hand side is $\delta$ applied to a simple coroot, hence it equals one. This completes the proof.
\end{proof}

\begin{lemma} Suppose that $n\le \mathbf n$. As in the statement and proof of Theorem \ref{root_factorization_g0} (with $w=1$), let $g^{(n)}=i_{\tau_n}(g(\zeta_n))..i_{\tau_1}(g(\zeta_1))$.
Then $l(g^{(n)})\in \dot N^-\cap (w'_n)^{-1}\dot N^+w'_n$ and the expression for Haar measure of this nilpotent group, in the $\zeta$ coordinates, is given by
\[
d\lambda_{\dot N^-\cap (w_n')^{-1}\dot N^+w_n'}(l(g^{(n)})) =  \prod_{k\le n}\mathbf a(\zeta_k)^{-2(\dot{\delta}(h_{\tau_k})-1)}\vert d\zeta_k\vert
\]
up to a normalization, where a sign is positive if and only if the corresponding root is of compact type.
\end{lemma}

\begin{proof} If $n=2$, then by (\ref{2ndcase1}),
\[
l(g^{(2)})=exp(\zeta_2 f_{\tau_2}+a(\zeta_2)^{-\tau_1(h_{\tau_2})}\zeta_1
f_{\tau_1}).
\]
Together with Lemma \ref{alglemma}, this completes the proof for $n=1,2$.

Now suppose $n>2$ and the result holds for $n-1$. We need to revisit how we obtained $l(g^{(n)})$, beginning after line (\ref{induction1}) in the induction step for the proof of
Theorem \ref{root_factorization_g0}.
Recall that
\begin{eqnarray*}
 g^{(n)} & = & \exp(\zeta_n f_{\tau_n})\mathbf a(\zeta_n)^{h_{\tau_n}}(\mathbf w'_{n-1})^{-1}\exp(\pm\bar{\zeta}_n e_{\gamma_n})\mathbf w'_{n-1}\exp(\zeta_{n-1}
f_{\tau_{n-1}})\widetilde l a^{(n-1)}u^{(n-1)} \\
 & = & \exp(\zeta_n f_{\tau_n})\mathbf a(\zeta_n)^{h_{\tau_n}}(\mathbf w'_{n-1})^{-1}\exp(\pm\bar{\zeta}_ne_{\gamma_n}) \widetilde u\mathbf w'_{n-1}a^{(n-1)}u^{(n-1)},
\end{eqnarray*}
where $\widetilde u= \mathbf w'_{n-1}\exp(\zeta_{n-1} f_{\tau_{n-1}})\widetilde l (\mathbf w'_{n-1})^{-1}\in w'_{n-1} N^{-} (w'_{n-1})^{-1} \cap
N^+$, and $\widetilde l \in \dot N^-\cap (w'_{n-2})^{-1} \dot N^{+} w'_{n-2}$. The first term will be the first factor in the ultimate expression for $l(g^{(n)})$. The conjugation by $a(\zeta_n)^{h_{\tau_n}}$  will affect volume,
and we will consider this below. The last term does not affect $l(g^{(n)})$. Consider the product of the other terms, which we rewrite as
\begin{eqnarray}
 &  & (\mathbf w'_{n-1})^{-1}\exp(\pm\bar{\zeta}_ne_{\gamma_n}) \widetilde u \mathbf w'_{n-1} \nonumber \\
 & = &(\mathbf w'_{n-1})^{-1}\exp(\pm\bar{\zeta}_ne_{\gamma_n}) \widetilde u\exp(\mp\bar{\zeta}_ne_{\gamma_n}) \exp(\pm\bar{\zeta}_ne_{\gamma_n}) \mathbf w'_{n-1} \label{rewrite1}
\end{eqnarray}
by inserting the identity.  Since conjugation by $\mathbf w'_{n-1}$ is a group isomorphism from the lower triangular nilpotent group
\[
\dot N^-\cap (w'_{n-1})^{-1}\dot N^+ w'_{n-1}=\exp(\sum_{j< n}\mathbb C f_{\tau_j})
\]
to the upper triangular nilpotent group $\dot N^+\cap w'_{n-1}\dot N^-(w'_{n-1})^{-1}$ (where $\widetilde u$ lives) the Haar measure for the first is pushed to the Haar measure for the second. We now consider the decomposition
\[
\dot N^+=\left(\dot N^+\cap w'_{n-1}\dot N^-(w'_{n-1})^{-1}\right)\left(\dot N^+\cap w'_{n-1}\dot N^+(w'_{n-1})^{-1}\right)
\]
and we write $u=(u)_1(u)_2$ for the corresponding factorization of elements $u\in \dot{N}^+$.

The key fact is that, if we set $u_0=\exp(\pm\bar{\zeta}_ne_{\gamma_n})$, then the map
\begin{equation}\label{volume_preserving_map}
\dot N^+\cap w'_{n-1}\dot N^-(w'_{n-1})^{-1}\to \dot N^+\cap w'_{n-1}\dot N^-(w'_{n-1})^{-1}\text{ defined by }u\mapsto \left(u_0 uu_0^{-1}\right)_1
\end{equation}
preserves the invariant volume. To see this, trivialize the tangent bundle for the nilpotent group $\dot N^+\cap w'_{n-1}\dot N^-(w'_{n-1})^{-1}$ using left translation, and fix $u$. Then derivative for the map (\ref{volume_preserving_map}) at $u$ is identified with the linear transformation
\[
\dot{\mathfrak n}^+\cap w'_{n-1}\dot{\mathfrak n}^-(w'_{n-1})^{-1}\to \dot{\mathfrak n}^+\cap w'_{n-1}\dot{\mathfrak n}^-(w'_{n-1})^{-1}
\]
given by
\begin{equation}\label{derivative_of_volume_preserving_map}
X\to \left.\frac{d}{dt}\right|_{t=0}\left(u_0 uu_0^{-1}\right)_1^{-1}\left(u_0ue^{tX}u_0^{-1}\right)_1.
\end{equation}
Now,
\begin{eqnarray*}
(u_0ue^{tX}u_0^{-1})_1 & = & (u_0uu_0^{-1}u_0e^{tX}u_0^{-1})_1  \\
& = & ((u_0uu_0^{-1})_1(u_0uu_0^{-1})_2u_0e^{tX}u_0^{-1})_1 \\
& = & (u_0uu_0^{-1})_1((u_0uu_0^{-1})_2u_0e^{tX}u_0^{-1}(u_0uu_0^{-1})_2^{-1})_1
\end{eqnarray*}
because the factor map $(\cdot)_1$ is equivariant for left multiplication by the first factor group $\dot{N}^+\cap w_{n-1}'\dot{N}^-(w_{n-1}')^{-1}$ and invariant for right multiplication by the second factor group $\dot{N}^+\cap w_{n-1}'\dot{N}^+(w_{n-1}')^{-1}$.  Consequently, (\ref{derivative_of_volume_preserving_map}) becomes
\begin{eqnarray*}
X & \to & \left.\frac{d}{dt}\right|_{t=0}\left(u_0 uu_0^{-1}\right)_1^{-1}\left(u_0ue^{tX}u_0^{-1}\right)_1 \\
& = & \left.\frac{d}{dt}\right|_{t=0}\left(\mathrm{Ad}((u_0uu_0^{-1})_2u_0)(e^{tX})\right)_1 \\
& = & \left(\mathrm{Ad}((u_0uu_0^{-1})_2u_0)(X)\right)_1
\end{eqnarray*}
where $(\cdot)_1$ now denotes the infinitesimal projection to $\dot{\n}^+\cap w_{n-1}'\dot{\n}^-(w_{n-1}')^{-1}$ in $\dot{\n}^+$.  Since $(u_0uu_0^{-1})_2u_0\in \dot{N}^+$, this is clearly the compression of a unipotent map on $\dot{\n}^+$.  Hence, its determinant is 1 at each point $u$.

Note that
$$\left(\exp(\pm\bar{\zeta}_ne_{\gamma_n}) \widetilde u\right)_1=\left(\exp(\pm\bar{\zeta}_ne_{\gamma_n}) \widetilde u\exp(\mp\bar{\zeta}_ne_{\gamma_n})\right)_1 \exp(\pm\bar{\zeta}_ne_{\gamma_n})$$
because $\exp(\pm\bar{\zeta}_ne_{\gamma_n})$ is in the second factor of the nilpotent group decomposition.
We finally have
$$l(g^{(n)}) = \exp(\zeta_n f_{\tau_n})\mathbf
l$$ where
\[
\mathbf l=\mathbf a(\zeta_n)^{h_{\tau_n}}\mathbf
w_{n-1}' \left(\exp(\pm\bar{\zeta}_ne_{\gamma_n}) \widetilde u\right)_1 (\mathbf w_{n-1}')^{-1}\mathbf
a(\zeta_n)^{-h_{\tau_n}}\in \dot N^- \cap (\mathbf w'_{n-1})^{-1} N^+
\mathbf w'_{n-1}.
\]
When we conjugate by the factor $\mathbf a(\zeta_n)^{h_{\tau_n}}$, we are multiplying the coefficient of $f_{\tau_j}$ by a factor $\mathbf a(\zeta_n)^{-\tau_j(h_{\tau_n})}$. Using the induction step, this implies that the Haar measure in the statement
of the Lemma is of the form
$$\left(\prod_{j=1}^{n-1}\mathbf a(\zeta_n)^{-2\tau_j(h_{\tau_n})}\right)\vert d\zeta_n\vert \left(\prod_{k< n}(1\pm\vert\zeta_k\vert^2)^{\delta(h_{\tau_k})-1}\vert d\zeta_k\vert \right)   $$

Applying Lemma \ref{alglemma} to the first product completes the proof.
\end{proof}

\end{document}